\documentclass[12pt]{amsart}
\pdfoutput=1
\usepackage{graphicx}
\usepackage{amssymb}
\usepackage{amsmath}
\usepackage{amsthm}
\usepackage{amscd}
\usepackage{epsfig}

\newtheorem{theorem}{Theorem}[section]

\newtheorem{corollary}[theorem]{Corollary}
\newtheorem{lemma}[theorem]{Lemma}

\newtheorem{proposition}[theorem]{Proposition}

\newtheorem{remark}[theorem]{Remark}
\theoremstyle{definition}

\numberwithin{equation}{section}

\newcommand{\N}{\mathbb N}

\begin{document}

\title{Topological entropy of sets of generic points for actions of amenable groups}

\author [Dongmei Zheng and Ercai Chen]{Dongmei Zheng and Ercai Chen}

\address{School of Mathematical Sciences and Institute of Mathematics, Nanjing Normal University, Nanjing 210023, Jiangsu, P.R.China
 \& Department of Applied Mathematics, College of Science, Nanjing Tech University, Nanjing 211816, Jiangsu, P.R. China} \email{dongmzheng@163.com}

\address{School of Mathematical Sciences and Institute of Mathematics, Nanjing Normal University, Nanjing 210023, Jiangsu, P.R.China} \email{ecchen@njnu.edu.cn}

\thanks{}


\begin{abstract}
Let $G$ be a countable discrete amenable group which acts continuously on a compact metric space $X$ and let $\mu$ be an ergodic $G-$invariant Borel probability measure on $X$.
For a fixed tempered F{\o}lner sequence $\{F_n\}$ in $G$ with $\lim\limits_{n\rightarrow+\infty}\frac{|F_n|}{\log n}=\infty$,
we prove the following variational principle:
$$h^B(G_{\mu},\{F_n\})=h_{\mu}(X,G),$$
where $G_{\mu}$ is the set of generic points for $\mu$ with respect to $\{F_n\}$ and $h^B(G_{\mu},\{F_n\})$ is the Bowen topological entropy (along $\{F_n\}$) on $G_{\mu}$.
This generalizes the classical result of Bowen in 1973.
\end{abstract}

\maketitle

\markboth{D. Zheng and E. Chen}{On the topological entropy of generic points for actions of amenable groups}

%%%%%%%%%%%%%%%%%%%%%%%%%%%%%%%%%%%%%%%%%%%%%%%%%%%%%%%%%%%%%%%%%%%%%%%%%%%%%%%%%%%%%%%%%%%%%%%%%%%%%%%%%%%%%%%%%%%%%%%%%%%%%%%%%%%%
%%%%%%%%%%%%%%%%%%%%%%%%%%%%%%%%%%%%%%%%%%%%%%%%%%%%%%%        Introduction      %%%%%%%%%%%%%%%%%%%%%%%%%%%%%%%%%%%%%%%%%%%%%%%%%%%
%%%%%%%%%%%%%%%%%%%%%%%%%%%%%%%%%%%%%%%%%%%%%%%%%%%%%%%%%%%%%%%%%%%%%%%%%%%%%%%%%%%%%%%%%%%%%%%%%%%%%%%%%%%%%%%%%%%%%%%%%%%%%%%%%%%%
\section{Introduction}

In 1973, Bowen \cite{B} introduced a definition of topological entropy of subsets for $\mathbb Z-$ or $\mathbb N-$action systems, which later on was known as the {\it Bowen topological entropy}.
The idea comes from resembling the definition of Hausdorff dimension to dynamical system.
This definition of entropy plays a key role in many aspects of ergodic theory and dynamical system, especially with the connection with dimension theory, statistical physics and multifractal analysis.

In that paper, Bowen proved the following results:
\begin{enumerate}
  \item[(A).] Bowen topological entropy of the whole system equals to the usual topological entropy.
  \item[(B).] If $\mu$ is an invariant Borel probability measure and $Y$ a subset with $\mu(Y)=1$, then the Bowen topological entropy of $Y$ is bigger than the
        measure-theoretic entropy with respect to $\mu$.
  \item[(C).] If in addition $\mu$ is ergodic, then the Bowen topological entropy of the set of generic points of $\mu$ is equal to the
        measure-theoretic entropy with respect to $\mu$.
\end{enumerate}

It is nature to know whether the above results still hold for dynamical systems with more general group actions.

Let $G$ be a countable discrete infinite group with the unit $e_G$. $G$ is said to be {\it amenable} if there exists a sequence of finite subsets $\{F_n\}$ of $G$ which is asymptotically invariant, i.e.,
$$\lim_{n\rightarrow+\infty}\frac{|F_n\vartriangle gF_n|}{|F_n|}=0, \text{ for all } g\in G.$$
Such sequences are called {\it F{\o}lner sequences}.

Throughout this paper, we let $(X,G)$  be a $G-$action topological dynamical system, where $X$ is a compact metric space with metric $d$ and $G$ a countable discrete infinite amenable group
acting on $X$ continuously. Denote by $M(X)$, $M(X,G)$ and $E(X,G)$ the collection of Borel probability measures,
$G$-invariant Borel probability measures and ergodic $G$-invariant Borel probability measures on $X$ respectively. Another equivalent definition for the group $G$ to be amenable is that
$M(X,G)$ is non-empty when $G$ acts continuously on every compact metric space $X$. One may refer to Ornstein and Weiss \cite{OW2} for more knowledge of amenable group actions.

Let $h_{top}(X,G)$ be the topological entropy of $(X,G)$. For $\mu\in M(X,G)$, let $h_{\mu}(X,G)$ be the measure-theoretic entropy of $(X,G)$ with respect to $\mu$.
For a subset $Y\subset X$ and a F{\o}lner sequence $\{F_n\}$ in $G$, let $h^B_{top}(Y,\{F_n\})$ be the Bowen topological entropy of $Y$. In section 2, we will give precise definitions
of these entropies.

Recall that a F{\o}lner sequence $\{F_n\}$ in $G$ is said to be {\it tempered} (see Shulman \cite{S}) if there exists a constant $C$ which is independent of $n$ such that
\begin{align}\label{tempered}
|\bigcup_{k<n}F_k^{-1}F_n|\le C|F_n|, \text{ for any }n.
\end{align}
In \cite{ZC}, Bowen's result (A) is proved to be true when the F{\o}lner sequence $\{F_n\}$ is tempered and satisfying certain increasing condition.
More precisely,
\begin{theorem}[Theorem 1.1 of \cite{ZC}]\label{th-1-1}\quad
Let $(X,G)$ be a compact metric $G-$action topological dynamical system and $G$ a discrete countable amenable group, then for any tempered F{\o}lner sequence $\{F_n\}$ in $G$
with the increasing condition
\begin{equation}\label{eq-1-2}
    \lim\limits_{n\rightarrow+\infty}\frac{|F_n|}{\log n}=\infty,
\end{equation}
$$h_{top}^B(X,\{F_n\})=h_{top}(X,G).$$
\end{theorem}

In this paper, by using different approach of Bowen's original proofs for $\mathbb Z$-actions, we will show that Bowen's results (B) and (C) also hold under the same condition as in Theorem \ref{th-1-1}. The statements are the following.

\begin{theorem}\label{th-1-2}\quad
Let $(X,G)$ and $\{F_n\}$ be as in Theorem \ref{th-1-1} and $\mu\in M(X,G)$. If $Y\subset X$ and $\mu(Y)=1$, then $h_{\mu}(X,G)\le h^B_{top}(Y, \{F_n\})$.
\end{theorem}

\begin{theorem}\label{th-1-3}\quad
Let $(X,G)$ and $\{F_n\}$ be as in Theorem \ref{th-1-1} and $\mu\in E(X,G)$. Let
\begin{align*}
  G_{\mu}=\{x\in X: \lim_{n\rightarrow\infty}\frac{1}{|F_n|}\sum_{g\in F_n}f(gx)=\int_Xf d\mu, \text{ for any }f\in C(X)\},
\end{align*}
the set of generic points for $\mu$ with respect to $\{F_n\}$,
then $$h^B_{top}(G_{\mu}, \{F_n\})=h_{\mu}(X,G).$$
\end{theorem}

We remark here that $G_{\mu}$ depends on the choice of the F{\o}lner sequence $\{F_n\}$ and $G_{\mu}$ may be an empty set when $\mu$ is non-ergodic.
We also note that for the proof of Theorem \ref{th-1-2}, we use a non-ergodic version of Brin-Katok's entropy formula (Theorem \ref{th-A-1}) and a variational principle
for Bowen topological entropy in \cite{ZC}. Theorem \ref{th-1-2} also gives a lower bound for $h^B_{top}(G_{\mu}, \{F_n\})$ in Theorem \ref{th-1-3}.
For the upper bound, we employ the ideas of Pfister and Sullivan \cite{PS}.

\section{Preliminaries}
\subsection{Topological entropy}

Let $\mathcal {U}$ be an open cover of $X$ and let $N(\mathcal{U})$ denote the minimal cardinality of subcovers of $\mathcal{U}$. The topological entropy of $\mathcal{U}$ is
$$h_{top}(G,\mathcal {U})=\lim_{n\rightarrow+\infty}\frac{1}{|F_n|}\log N\big(\mathcal{U}^{F_n}\big),$$
where $\mathcal{U}^{F_n}=\bigvee_{g\in F_n}g^{-1}\mathcal{U}$.
It is shown that $h_{top}(G,\mathcal {U})$ does not depend on the choice of the F{\o}lner sequences $\{F_n\}$ (see \cite{HYZ}).
The {\it topological entropy} of $(X,G)$ is then defined by
$$h_{top}(X,G)=\sup_{\mathcal U} h_{top}(G,\mathcal {U}),$$
where the supremum is taken over all the open covers of $X$.

\subsection{Bowen topological entropy}
For a finite subset $F$ in $G$, $\varepsilon>0$ and a point $x$ in $X$, we denote the {\it Bowen ball} associated to $F$ with center $x$ and radius $\varepsilon$ by
\begin{align*}
B_{F}(x,\varepsilon)&=\{y\in X: d_{F}(x,y)<\varepsilon\}\\
&=\{y\in X: d(gx,gy)<\varepsilon, \text{ for any }g\in F\}.
\end{align*}

For $Z\subseteq X, s\ge0, N\in\mathbf{N}$, $\{F_n\}$ a F{\o}lner sequence in $G$ and $\varepsilon> 0$, define
$$\mathcal{M}(Z,N,\varepsilon,s,\{F_n\}) = \inf\sum_i\exp(-s|F_{n_i}|),$$
where the infimum is taken over all finite or countable families $\{B_{F_{n_i}}(x_i,\varepsilon)\}$ such that
$x_i\in X, n_i\ge N$ and $\bigcup_i B_{F_{n_i}}(x_i,\varepsilon)\supseteq Z$. The quantity $\mathcal{M}(Z,N,\varepsilon,s,\{F_n\})$
does not decrease as $N$ increases and $\varepsilon$ decreases, hence the following limits exist:
$$\mathcal{M}(Z,\varepsilon,s,\{F_n\})=\lim_{N\rightarrow+\infty}\mathcal{M}(Z,N,\varepsilon,s,\{F_n\}),\mathcal{M}(Z,s,\{F_n\})=\lim_{\varepsilon\rightarrow0}\mathcal{M}(Z,\varepsilon,s,\{F_n\}).$$
The {\it Bowen topological entropy} $h^B_{top}(Z,\{F_n\})$ is then defined as the critical value
of the parameter $s$, where $\mathcal{M}(Z,s,\{F_n\})$ jumps from $+\infty$ to $0$, i.e.,
\begin{equation*}
\mathcal{M}(Z,s,\{F_n\})=\begin{cases}
                             0, s > h^B_{top}(Z,\{F_n\}),\\
                             +\infty, s < h^B_{top}(Z,\{F_n\}).
                        \end{cases}
\end{equation*}

From the above definition, it is easy to check that for each $s\ge 0$, $\mathcal{M}(\cdot,s,\{F_n\})$ is an outer measure on $X$.
Then Bowen topological entropy satisfies the following properties.
\begin{proposition}\label{prop-2-1}\quad
  \begin{enumerate}
    \item If $Z_1\subset Z_2\subset X$, then $h^B_{top}(Z_1,\{F_n\})\le h^B_{top}(Z_2,\{F_n\})$.
    \item If $Y_i\subset X$ for $i=1,2,\dots$, then $h^B_{top}(\bigcup_{i=1}^{\infty}Y_i,\{F_n\})=\sup_i h^B_{top}(Y_i,\{F_n\})$.
  \end{enumerate}
\end{proposition}

\subsection{Measure-theoretic entropy}

Let $(X,G,\mu)$ be a $G$-measurable dynamical system where $(X,\mathcal{B},\mu)$ is a probability space and $G$ a group that acts in a measure preserving fashion on $(X,\mathcal{B},\mu)$.
Let $\mathcal {P}$ be a finite measurable partition of $X$. For a finite subset $F$ in $G$, we denote by $\mathcal{P}^F=\bigvee_{g\in F}g^{-1}\mathcal{P}$.
When $G$ is a countable discrete amenable group, the measure-theoretic entropy of $\mathcal {P}$ (with respect to $\mu$) is defined by
$$h_{\mu}(G,\mathcal{P})=\lim_{n\rightarrow+\infty}\frac{1}{|F_n|}H_{\mu}(\mathcal{P}^{F_n}),$$
where $\{F_n\}$ is any F{\o}lner sequence in $G$ and the definition is independent of the specific F{\o}lner sequence $\{F_n\}$ (see, for example, \cite{OW2}).
The {\it measure-theoretic entropy} of the system $(X,G,\mu)$, $h_{\mu}(X,G)$, is the supremum of $h_{\mu}(G,\mathcal{P})$ over $\mathcal {P}$.

Consider the $\sigma-$algbra $\mathcal{I}_{\mu}=\{A\in \mathcal{B}:\mu(A\triangle g^{-1}A)=0, \forall g\in G\}$.
Let $p:X\rightarrow X/\mathcal{I}_{\mu}=Y$ be the associated projection and $\mu=\int_Y \mu_yd \pi(y)$ be the decomposition of $\mu$ over $Y$.
Such a decomposition is called the {\it ergodic decomposition} of $\mu$, since for each $y\in Y$,
$p^{-1}(y)$ is $G$-invariant and $(p^{-1}(y),G,\mu_y)$ is a $G$-ergodic measurable dynamical system.

For a measurable partition $\mathcal{P}$ and $x\in X$, denote by $\mathcal{P}(x)$ the element in $\mathcal{P}$ which $x$ belongs to.
The following is the non-ergodic version of Shannon-McMillan-Breiman theorem for amenable group actions. For the ergodic case, one may also see \cite{L,OW}.

\begin{theorem}[SMB Theorem, Theorem 6.2 of \cite{W}]\label{th-2-2}\quad
Let $(X,G,\mu)$ be a $G$-measure preserving system and $G$ a countable discrete amenable group. Then for any tempered F{\o}lner sequence $\{F_n\}$ in $G$
 with the increasing condition \eqref{eq-1-2} and any finite measurable partition $\mathcal{P}$ one has that for $\mu$-a.e. $x\in X$,
 $$\lim_{n\rightarrow +\infty}-\frac{1}{|F_n|}\log \mu(\mathcal{P}^{F_n}(x))=h_{\mu_y}(G,\mathcal{P}|p^{-1}(y))\triangleq h(x,\mathcal{P}),$$
 where $y\in Y$ such that $p^{-1}(y)$ is the ergodic component containing $x$ and $$\int_Xh(x,\mathcal{P})d\mu(x)=h_{\mu}(G,\mathcal{P}).$$
\end{theorem}

\section{Brin-Katok's entropy formula for non-ergodic case}
  In this section, we will prove Brin-Katok's entropy formula \cite{BK} for amenable group action dynamical systems. The statement of this formula is the following.
\begin{theorem}[Brin-Katok's entropy formula: non-ergodic case]\label{th-A-1}
Let $(X,G)$ be a compact metric $G-$action topological dynamical system and $G$ a countable discrete amenable group.
Let $\mu\in M(X,G)$ and $\{F_n\}$ a tempered F{\o}lner sequence in $G$
 with the increasing condition \eqref{eq-1-2}, then for $\mu$ almost every $x\in X$,
\begin{align*}
&\lim_{\delta\rightarrow 0}\liminf_{n\rightarrow+\infty}-\frac{1}{|F_n|}\log \mu(B_{F_n}(x,\delta)) \\
=&\lim_{\delta\rightarrow 0}\limsup_{n\rightarrow+\infty}-\frac{1}{|F_n|}\log \mu(B_{F_n}(x,\delta))\triangleq h_{\mu}(x),
\end{align*}
where $h_{\mu}(x)$ is a $G-$invariant measurable function such that $\int_X h_{\mu}(x)d\mu=h_{\mu}(X,G)$.
\end{theorem}
For the proof, we follow the proof originally due to Brin and Katok \cite{BK} for $\mathbb Z$-actions.

Let $\mu=\int_Y \mu_yd \pi(y)$ be the $G-$ergodic decomposition of $\mu$ and $p:X\rightarrow Y$ be the associated projection.
For each $y\in Y$, let $h(y)=h_{\mu_y}(p^{-1}(y),G)$ be the measure-theoretic entropy restricted to the system $(p^{-1}(y),G,\mu_y)$.
For any $M>0$, denote by $X_M=p^{-1}(h^{-1}([0,M)))$ and $X_M'=p^{-1}(h^{-1}([M,\infty)))$. Let $X_{\infty}=p^{-1}(h^{-1}(\infty))$.
Then $X=X_M\bigcup X_M'\bigcup X_{\infty}$.

\begin{lemma}\label{lem-2}\quad
\begin{enumerate}
  \item For any $M>0$,
  $$\int_{X_M} \lim\limits_{\delta\rightarrow 0}\liminf\limits _{n\rightarrow+\infty}-\frac{1}{|F_n|}\log \mu(B_{F_n}(x,\delta))d\mu \ge \int_{h^{-1}([0,M))}h(y)d\pi(y).$$
  \item For $\mu$ almost every $x\in X_{\infty}$, $$\lim\limits_{\delta\rightarrow 0}\liminf\limits _{n\rightarrow+\infty}-\frac{1}{|F_n|}\log \mu(B_{F_n}(x,\delta))=\infty.$$
\end{enumerate}

\end{lemma}
\begin{proof}
Obviously (1) holds if $\mu(X_M)=0$ and (2) holds if $\mu(X_{\infty})=0$. So we may assume that both $\mu(X_M)$ and $\mu(X_{\infty})$ are positive.

Take $L\in \mathbb{N}$ to be sufficiently large and let $\gamma=\frac{M}{L}$.
For $l=0,1,\cdots,L-1$, let $A_l=p^{-1}(h^{-1}([l\gamma,(l+1)\gamma)))$ and let $A_{\infty}=X_{\infty}$.

For a finite measurable partition of $X$, say $\beta$, denote by ${\rm diam}(\beta)=\max_{B\in \beta}{\rm diam}(B)$ and $\partial\beta=\bigcup_{B\in\beta}\partial B$.

Let $\eta_m$ be a sequence of finite measurable partition of $X$ with $\lim_{m\rightarrow \infty}{\rm diam}(\eta_m)=0$ and $\mu(\partial\eta_m)=0$ for each $m$.
Then $$\lim_{m\rightarrow\infty}h_{\nu}(G,\eta_m)=h_{\nu}(X,G),\text{ for any }\nu\in M(X,G).$$
By the SMB theorem, for $\mu$-a.e. $x\in X$,
$$\lim_{n\rightarrow \infty}-\frac{1}{|F_n|}\log \mu(\eta_m^{F_n}(x))\triangleq h(x,\eta_m)=h_{\mu_y}(G,\eta_m|p^{-1}(y))=h_{\mu_y}(G,\eta_m),$$
where $p^{-1}(y)$ is the ergodic component that contains $x$, i.e. $p(x)=y$.
Hence for $\mu$-a.e. $x\in X$, $\lim_{m\rightarrow\infty}h(x,\eta_m)=h_{\mu_y}(p^{-1}(y),G)=h(y)$, where $y=p(x)$.

For any $\varepsilon>0$, by Egorov's Theorem, we then can choose $\eta=\eta_m$ for $m$ sufficiently large such that up to a subset of $X$ with small $\mu$ measure (say, less than $\varepsilon$), it holds that
$h(x,\eta)>\min \{\frac{1}{\varepsilon}, h(p(x))-\varepsilon\}$.
Hence there exists sufficiently large $N_2$, whence $n>N_2$, for each $l=0,1,\cdots,L-1$,
\begin{align}\label{set-2}
 \mu(\{x\in A_l: \forall n'\ge n, -\frac{1}{|F_{n'}|}\log \mu(\eta^{F_{n'}}(x))>l\gamma-2\varepsilon\})>\mu(A_l)-2\varepsilon,
\end{align}
and
\begin{align}\label{set-2'}
 \mu(\{x\in A_{\infty}: \forall n'\ge n, -\frac{1}{|F_{n'}|}\log \mu(\eta^{F_{n'}}(x))>\frac{1}{\varepsilon}-2\varepsilon\})>\mu(A_{\infty})-2\varepsilon.
\end{align}

For $\delta>0$, we define $U_{\delta}=\bigcup_{C\in\eta}\big(\bigcup_{x\in C}B(x,\delta)\setminus C\big)$.
Note that for each $C\in \eta$, $\bigcap_{\eta}\big(\bigcup_{x\in C}B(x,\delta)\setminus C\big)\subset \partial C$. For any sufficiently small $\varepsilon>q>0$, since $\mu(\partial\eta)=0$, we can find $\delta>0$ which is sufficiently small such that $\mu(U_{\delta})$ is less than $q^2$.
Applying the pointwise ergodic theorem (see e.g. Theorem 3.3 of \cite{L}) to the function $\chi_{U_{\delta}}$, for a.e. $x\in X$,
$$\frac{1}{|F_n|}\sum\limits_{g\in F_n}\chi_{U_{\delta}}(gx)\rightarrow f_{U_{\delta}}(x),$$
where $f_{U_{\delta}}(x)\in L^1_{\mu}(X)$ and
$\int_Xf_{U_{\delta}}(x)d\mu=\mu(U_{\delta})$.
Then there exists sufficiently large $N_1$, whence $n>N_1$,
\begin{align}\label{set-1}
 \mu(\{x\in X: &\forall n'\ge n, \sum\limits_{g\in F_{n'}}\chi_{U_{\delta}}(gx)<q|F_{n'}|\})>1-\varepsilon.
\end{align}

Let $E_l$ (resp. $E_{\infty}$) be the intersection of the sets in the left-hand side of \eqref{set-2} (resp. \eqref{set-2'}) and \eqref{set-1}.
Then for any $n>\max\{N_1,N_2\}$, $\mu(E_l)>\mu(A_l)-3\varepsilon$ for each $l=0,1,\cdots,L-1,\infty$.

Let $w_{\eta,F_n}(x)=(\eta(gx))_{g\in F_{n}}$ be the $(\eta,F_n)-$name of $x$. For any $y\in B(x,\delta)$, we have that either $\eta(x)=\eta(y)$ or $x\in U_{\delta}$.
Hence for each $l=0,1,\dots,L-1,\infty$, if $x\in E_l$ and $y\in B_{F_n}(x,\delta)$, then the Hamming distance between $w_{\eta,F_n}(x)$ and $w_{\eta,F_n}(y)$ is less than $q$. This implies that
whence $x\in E_l$, $$B_{F_n}(x,\delta)\subset \bigcup \{\eta_{F_n}(y): w_{\eta,F_n}(y) \text{ is }q-\text{close to }w_{\eta,F_n}(x)\text{ under Hamming metric}\}.$$

By Stirling's formula, there exists $N_3$ sufficiently large such that whence $n>N_3$, the total number of such $(\eta,F_n)-$names, denoted by $L_n$, can be estimated by:
\begin{align*}
 L_n\le\sum_{j=0}^{\lfloor q|F_n|\rfloor}\binom{|F_n|}{j}(\#\eta-1)^j\le \exp(K|F_n|),
\end{align*}
where $K$ can be chosen as
$$K=q+q\log(\#\eta-1)-q\log q-(1-q)\log(1-q).$$
For the calculation of $K$, one may refer to \cite{K} or \cite{BK}.

We now note that $K$ is a constant only dependent on $\#\eta$, $\varepsilon$ and $q$ but independent of $x$ and $n$.
Moreover, when $\gamma$ and $\varepsilon$ are fixed (hence $\#\eta$), we can choose $q$ small enough such that $K$ tends to $0$ while $q$ tends to $0$.
Hence we can make $$L_n\le \exp(\varepsilon|F_n|).$$

For $l=0,1,\cdots,L-1$, let $$D_{l,n}=\{x\in E_l: \mu(B_{F_n}(x,\delta))>\exp((-l\gamma+5\varepsilon)|F_n|)\}.$$
And let
$$D_{\infty,n}=\{x\in E_{\infty}: \mu(B_{F_n}(x,\delta))>\exp((-\frac{1}{\varepsilon}+5\varepsilon)|F_n|)\}.$$

To prove (1), we consider the case for $l=0,1,\cdots,L-1$.

If we can prove that $\sum_{n=N}^{\infty}\mu(D_{l,n})<\infty$, then apply the Borel-Cantelli Lemma:
for a.e. $x\in E_l$,
\begin{align}\label{ineq-3-10}
  \liminf_{n\rightarrow+\infty}-\frac{1}{|F_n|}\log \mu(B_{F_n}(x,\delta))\ge l\gamma-5\varepsilon.
\end{align}
Hence we can obtain that
\begin{align*}
&\int_{X_M}\liminf_{n\rightarrow+\infty}-\frac{1}{|F_n|}\log \mu(B_{F_n}(x,\delta))d\mu\\
\ge &\sum_{l=0}^{L-1}l\gamma\mu(E_l)-5\varepsilon\\
= &\sum_{l=0}^{L-1}l\gamma\mu(A_l)-\sum_{l=0}^{L-1}l\gamma(\mu(A_l)-\mu(E_l))-5\varepsilon\\
\ge &\int_{h^{-1}([0,M))}h(y)d\pi(y)-\gamma-\frac{1}{2}L(L-1)\gamma3\varepsilon-5\varepsilon.
\end{align*}
Let $\varepsilon$ go to $0$ first (this makes $\delta$ tending to $0$) and then let $\gamma$ go to $0$ (by letting $L$ tend to infinity),
$$\int_{X_M} \lim\limits_{\delta\rightarrow 0}\liminf\limits _{n\rightarrow+\infty}-\frac{1}{|F_n|}\log \mu(B_{F_n}(x,\delta))d\mu \ge \int_{h^{-1}([0,M))}h(y)d\pi(y).$$

Now we estimate the measures of $D_{l,n}$'s.

For any $x\in D_{l,n}$, in those $L_n-$many $(\eta,F_n)-$names which are $q-$close to $w_{\eta,F_n}(x)$ in Hamming distance,
there exists at least one corresponding atom of $\eta_{F_n}$ whose measure is greater than $\exp((-l\gamma+4\varepsilon)|F_n|)$.
The total number of such atoms will not exceed
$\exp((l\gamma-4\varepsilon)|F_n|)$. Hence $Q_{l,n}$, the total number of elements in $\eta^{F_n}$ that intersect $D_{l,n}$, satisfies:
$$Q_{l,n}\le L_n\exp((l\gamma-4\varepsilon)|F_n|)\le \exp((l\gamma-3\varepsilon)|F_n|).$$
Let $S_{l,n}$ denote the total measure of such $Q_{l,n}$ elements of $\eta^{F_n}$ whose intersections with $E_l$ have positive measure.
Then from \eqref{set-2},
$$S_{l,n}\le Q_{l,n}\exp((-l\gamma+2\varepsilon)|F_n|)\le\exp(-\varepsilon|F_n|),$$
which follows that
$$\mu(D_{l,n})\le S_{l,n}\le \exp(-\varepsilon|F_n|).$$

From the increasing condition \eqref{eq-1-2}, for sufficiently large $N_4$, whenever $n\ge N_4$, $\frac{|F_n|}{\log n}\ge \frac{2}{\varepsilon}$ holds,
which implies that $\exp(-\varepsilon|F_n|)\le n^{-2}$. And hence $\sum_{n=1}^{\infty}\mu(D_{l,n})<\infty$.

To prove (2), we need estimate the measures of $D_{\infty,n}$'s.

In the above treatment for $D_{l,n}$'s, replacing $l\gamma$ (resp. $D_{l,n}$'s, $Q_{l,n}$'s and $S_{l,n}$'s) by
$\frac{1}{\varepsilon}$ (resp. $D_{\infty,n}$'s, $Q_{\infty,n}$'s and $S_{\infty,n}$'s),
it also holds that $\sum_{n=N}^{\infty}\mu(D_{\infty,n})<\infty$, then apply the Borel-Cantelli Lemma again:
for a.e. $x\in E_{\infty}$,
\begin{align}\label{ineq-3-10'}
  \liminf_{n\rightarrow+\infty}-\frac{1}{|F_n|}\log \mu(B_{F_n}(x,\delta))\ge \frac{1}{\varepsilon}-5\varepsilon.
\end{align}

Letting $\varepsilon$  go to $0$, we then have for $\mu$ almost every $x\in X_{\infty}$,
$$\lim\limits_{\delta\rightarrow 0}\liminf\limits _{n\rightarrow+\infty}-\frac{1}{|F_n|}\log \mu(B_{F_n}(x,\delta))=\infty.$$
\end{proof}

Now we can finish the proof of Theorem \ref{th-A-1}.
\begin{proof}[Proof of Theorem \ref{th-A-1}]

Let $\delta>0$ be given and let $\xi_{\delta}$ be a finite measurable partition of $X$ such that the diameter of every set in $\xi_{\delta}$ is less than $\delta$.
Then by the SMB theorem for amenable group actions, for $\mu$-a.e. $x\in X$,
$$\lim_{n\rightarrow \infty}-\frac{1}{|F_n|}\log \mu({\xi_{\delta}}^{F_n}(x))\triangleq h(x,\xi_{\delta})=h_{\mu_y}(G,\xi_{\delta}|p^{-1}(y)),$$
where $y=p(x)$. Hence for any $M>0$,
$$\int_{X_M} h(x,\xi_{\delta})d\mu=\int_{h^{-1}([0,M))}h_{\mu_y}(G,\xi_{\delta}|p^{-1}(y))d\pi(y)\le \int_{h^{-1}([0,M))}h(y)d\pi(y).$$
Since ${\xi_{\delta}}^{F_n}(x)\subset B_{F_n}(x,\delta)$, we have that
\begin{align}\label{ineq-3-8}
  &\int_{X_M}\lim_{\delta\rightarrow 0}\limsup_{n\rightarrow+\infty}-\frac{1}{|F_n|}\log \mu(B_{F_n}(x,\delta))d\mu\\
  \le &\int_{X_M} \lim_{\delta\rightarrow 0}h(x,\xi_{\delta})d\mu=\int_{h^{-1}([0,M))}h(y)d\pi(y).\nonumber
\end{align}

Together with (1) of Lemma \ref{lem-2}, we have that
\begin{align}\label{ineq-3-9}
  &\int_{X_M}\lim_{\delta\rightarrow 0}\limsup_{n\rightarrow+\infty}-\frac{1}{|F_n|}\log \mu(B_{F_n}(x,\delta))d\mu\\
  =&\int_{X_M}\lim_{\delta\rightarrow 0}\liminf_{n\rightarrow+\infty}-\frac{1}{|F_n|}\log \mu(B_{F_n}(x,\delta))d\mu \nonumber \\
  =&\int_{h^{-1}([0,M))}h(y)d\pi(y)<\infty,\nonumber
\end{align}
which implies that for $\mu-$a.e. $x\in X_M$,
\begin{align*}
\lim_{\delta\rightarrow 0}\liminf_{n\rightarrow+\infty}-\frac{1}{|F_n|}\log \mu(B_{F_n}(x,\delta))=\lim_{\delta\rightarrow 0}\limsup_{n\rightarrow+\infty}-\frac{1}{|F_n|}\log \mu(B_{F_n}(x,\delta)).
\end{align*}

By (2) of Lemma \ref{lem-2}, for $\mu-$a.e. $x\in X_{\infty}$,
\begin{align*}
\lim_{\delta\rightarrow 0}\limsup_{n\rightarrow+\infty}-\frac{1}{|F_n|}\log \mu(B_{F_n}(x,\delta))
=\lim_{\delta\rightarrow 0}\liminf_{n\rightarrow+\infty}-\frac{1}{|F_n|}\log \mu(B_{F_n}(x,\delta))=\infty.
\end{align*}

Let $M$ tend to $\infty$, then $\mu(X_M\bigcup X_{\infty})$ tends to $1$. Hence for $\mu-$a.e. $x\in X$,
\begin{align*}
&\lim_{\delta\rightarrow 0}\liminf_{n\rightarrow+\infty}-\frac{1}{|F_n|}\log \mu(B_{F_n}(x,\delta)) \\
=&\lim_{\delta\rightarrow 0}\limsup_{n\rightarrow+\infty}-\frac{1}{|F_n|}\log \mu(B_{F_n}(x,\delta))\triangleq h_{\mu}(x).
\end{align*}
and $\int_X h_{\mu}(x)d\mu=h_{\mu}(G,X)$.

By \eqref{ineq-3-8} and \eqref{ineq-3-9}, for any $M>0$,
\begin{align*}
  \int_{h^{-1}([0,M))}h(y)d\pi(y)=&\int_{X_M}\lim_{\delta\rightarrow 0}\limsup_{n\rightarrow+\infty}-\frac{1}{|F_n|}\log \mu(B_{F_n}(x,\delta))d\mu\\
  \le & \int_{X_M} \liminf_{\delta\rightarrow 0}h(x,\xi_{\delta})d\mu
  \le \liminf_{\delta\rightarrow 0}\int_{X_M} h(x,\xi_{\delta})d\mu\\
  \le &\int_{h^{-1}([0,M))}h(y)d\pi(y)<\infty.
\end{align*}
Hence
\begin{align*}
h_{\mu}(x)=\liminf_{\delta\rightarrow 0}h(x,\xi_{\delta}), \text{ for }\mu-\text{a.e. }x\in X_M.
\end{align*}
Since $h(x,\xi_{\delta})$, $X_M$ and $X_{\infty}$ are all $G-$invariant, letting $M$ tend to infinity, $h_{\mu}(x)$ is also $G-$invariant on the whole $X$.
\end{proof}

\section{Proof of Theorem \ref{th-1-2}}

Let $(X,G)$ be a compact metric $G-$action topological dynamical system and $G$ a countable discrete amenable group.
For any $\mu\in M(X)$, $x\in X,n\in\mathbf{N}$, $\varepsilon>0$ and any F{\o}lner sequence $\{F_n\}$,
denote by
$$\underline h_{\mu}^{loc}(x,\varepsilon,\{F_n\})=\liminf_{n\rightarrow +\infty}-\frac{1}{|F_n|}\log \mu(B_{F_n}(x,\varepsilon)).$$
Then the lower local entropy of $\mu$ at $x$ (along $\{F_n\}$) is defined by
$$\underline h_{\mu}^{loc}(x,\{F_n\})=\lim_{\varepsilon\rightarrow 0}\underline h_{\mu}^{loc}(x,\varepsilon,\{F_n\})$$
and the lower local entropy of $\mu$ is defined by $$\underline h_{\mu}^{loc}(\{F_n\})=\int_X\underline h_{\mu}^{loc}(x,\{F_n\})d\mu.$$
Similarly, we can define the upper local entropy.

In \cite{ZC}, the authors proved the following variational principle between the lower local entropy and the Bowen entropy of compact subsets.

\begin{theorem}[Theorem 3.1 of \cite{ZC}]\label{th-3-2}\quad
Let $(X,G)$ be a compact metric $G-$action topological dynamical system and $G$ a discrete countable amenable group. If $K\subseteq X$ is non-empty and compact
and $\{F_n\}$ a sequence of finite subsets in $G$ with the increasing condition $\lim\limits_{n\rightarrow+\infty}\frac{|F_n|}{\log n}=\infty$,
then
$$h^B_{top}(K, \{F_n\})=\sup\{\underline h_{\mu}^{loc}(\{F_n\}):\mu(K)=1\},$$
where the supremum is taken over $\mu\in M(X)$.
\end{theorem}

With the help of the above theorem, we can now give the proof of Theorem \ref{th-1-2}.

\begin{proof}[Proof of Theorem \ref{th-1-2}.]
Let $\mu\in M(X,G)$ and $Y$ a subset of $X$ with $\mu(Y)=1$.
Let $\{Y_n\}_{n\in\N}$ be an increasing sequence of compact subsets of $Y$ such that $\mu(Y_n)>1-\frac{1}{n}$ for each $n\in\N$.

Then by Proposition \ref{prop-2-1},
\begin{align}\label{eq-3-1}
  h^B_{top}(Y,\{F_n\})\ge h^B_{top}(\bigcup_{n\in\N} Y_n,\{F_n\})=\lim_{n\rightarrow \infty}h^B_{top}(Y_n,\{F_n\}).
\end{align}
Denote by $\mu_n$ the restriction of $\mu$ on $Y_n$, i.e. for any $\mu-$measurable set $A\subset X$,
$$\mu_n(A)=\frac{\mu(A\cap Y_n)}{\mu(Y_n)}.$$
Applying Theorem \ref{th-3-2},
\begin{align}\label{eq-3-2}
  h^B_{top}(Y_n,\{F_n\})&=\sup \{\underline h_{\nu}^{loc}(\{F_n\}):\nu\in M(X),\nu(Y_n)=1\}\nonumber \\
  &\ge \underline h_{\mu_n}^{loc}(\{F_n\}).
\end{align}
Note that
\begin{align*}
  \underline h_{\mu_n}^{loc}(\{F_n\})&=\int_{Y_n}\lim_{\varepsilon\rightarrow 0}\liminf_{m\rightarrow +\infty}-\frac{1}{|F_m|}\log \mu_n(B_{F_m}(x,\varepsilon))d\mu_n\\
  &=\frac{1}{\mu(Y_n)}\int_{Y_n}\lim_{\varepsilon\rightarrow 0}\liminf_{m\rightarrow +\infty}-\frac{1}{|F_m|}\log \frac{\mu(B_{F_m}(x,\varepsilon)\cap Y_n)}{\mu(Y_n)}d\mu\\
  &\ge \frac{1}{\mu(Y_n)}\int_{Y_n}\lim_{\varepsilon\rightarrow 0}\liminf_{m\rightarrow +\infty}-\frac{1}{|F_m|}\log \frac{\mu(B_{F_m}(x,\varepsilon))}{\mu(Y_n)}d\mu\\
  &= \frac{1}{\mu(Y_n)}\int_{Y_n}\lim_{\varepsilon\rightarrow 0}\liminf_{m\rightarrow +\infty}-\frac{1}{|F_m|}\log \mu(B_{F_m}(x,\varepsilon))d\mu.
\end{align*}
By Theorem \ref{th-A-1},
$$\int_Y \lim_{\varepsilon\rightarrow 0}\liminf_{m\rightarrow +\infty}-\frac{1}{|F_m|}\log \mu(B_{F_m}(x,\varepsilon))=h_{\mu}(X,G).$$
Hence
\begin{align*}
  \lim_{n\rightarrow \infty}\underline h_{\mu_n}^{loc}(\{F_n\})\ge h_{\mu}(X,G).
\end{align*}
Together with \eqref{eq-3-1} and \eqref{eq-3-2},
$$h_{\mu}(X,G)\le h^B_{top}(Y, \{F_n\}).$$
\end{proof}

Noticing that $\mu(G_{\mu})=1$ for $\mu\in E(X,G)$, by Theorem \ref{th-1-2}, we have the following corollary.

\begin{corollary}\label{coro-3-2}\quad
Let $(X,G)$ and $\{F_n\}$ be as in Theorem \ref{th-1-1} and $\mu\in E(X,G)$,
then $$h_{\mu}(X,G)\le h^B_{top}(G_{\mu}, \{F_n\}).$$
\end{corollary}

\begin{remark}\quad
In general, when $\mu$ is non-ergodic, $G_{\mu}$ may not have full $\mu$ measure. In fact, there exist examples that
$h_{\mu}(X,G)>0$ while $G_{\mu}=\emptyset$ for $\mathbb Z-$actions. Hence $h_{\mu}(X,G)\le h^B_{top}(G_{\mu},\{F_n\})$ may not hold.
\end{remark}

\section{Proof of Theorem \ref{th-1-3}}

In this section, we will show the proof of Theorem \ref{th-1-3}.
Corollary \ref{coro-3-2} gives the lower bound. For the upper bound, we use the ideas of Pfister and Sullivan \cite{PS}.

For $\mu\in E(X,G)$, let $\{K_m\}_{m\in\N}$ be a decreasing sequence of closed convex neighborhoods of $\mu$ in $M(X)$ and let
$$A_{n,m}=\{x\in X:\frac{1}{|F_n|}\sum_{g\in F_n}\delta_x\circ g^{-1}\in K_m\}, \text{ for }m,n\in\N.$$
Then for any $m,N\ge 1$, $G_{\mu}\subset \bigcup_{n\ge N}A_{n,m}$.

Let $\varepsilon>0$, $F$ a finite subset of $G$. A subset $E\subset X$ is said to be {\it $(F,\varepsilon)$-separated},
if for any $x,y\in E$ with $x\neq y$, $d_F(x,y)>\varepsilon$.

Denote by $N(A_{n,m},n,\varepsilon)$ the maximal cardinality of any $(F_n,\varepsilon)$-separated subset of $A_{n,m}$.

  {\bf Claim.} $$\lim_{\varepsilon\rightarrow 0}\lim_{m\rightarrow\infty}\limsup_{n\rightarrow\infty}\frac{1}{|F_n|}\log N(A_{n,m},n,\varepsilon)\le h_{\mu}(X,G).$$
\begin{proof}[Proof of the claim]
If not, suppose that
\begin{align*}
  \lim_{\varepsilon\rightarrow 0}\lim_{m\rightarrow\infty}\limsup_{n\rightarrow\infty}\frac{1}{|F_n|}\log N(A_{n,m},n,\varepsilon)> h_{\mu}(X,G)+\delta,
\end{align*}
for some $\delta>0$. Then there exist $\varepsilon_0>0$ and $M\in \N$ such that for any $0<\varepsilon<\varepsilon_0$ and any $m\ge M$, it holds that
\begin{align*}
  \limsup_{n\rightarrow\infty}\frac{1}{|F_n|}\log N(A_{n,m},n,\varepsilon)> h_{\mu}(X,G)+\delta.
\end{align*}
Hence we can find a sequence $\{m(n)\}$ with $m(n)\rightarrow \infty$ such that
\begin{align*}
  \limsup_{n\rightarrow\infty}\frac{1}{|F_n|}\log N(A_{n,m(n)},n,\varepsilon)\ge h_{\mu}(X,G)+\delta.
\end{align*}
Now let $E_n$ be a $(F_n,\varepsilon)$-separated set of $A_{n,m(n)}$ with maximal cardinality and define
\begin{align*}
    \sigma_{n}=\frac{1}{\#E_n}\sum_{x\in E_n}\delta_x \text{ and }\mu_{n}=\frac{1}{|F_n|}\sum_{g\in F_n}\sigma_n\circ g^{-1}.
\end{align*}
Since
$$\frac{1}{|F_n|}\sum_{g\in F_n}\delta_x\circ g^{-1}\in K_{m(n)}, \text{ for any }x\in E_n$$
and $$\mu_n=\frac{1}{\#E_n}\sum_{x\in E_n}\frac{1}{|F_n|}\sum_{g\in F_n}\delta_x\circ g^{-1},$$
by the convexity of $K_{m}$'s, $\mu_n\in K_{m(n)}$. And hence $\mu_n\rightarrow \mu$ as $n$ goes to infinity.

Let $\beta$ be a finite Borel partition of $X$ such that ${\rm diam}(\beta)<\varepsilon$ and $\mu(\partial \beta)=0$.
Then each element of $\beta^{F_n}$ contains at most one point in $E_n$. Hence
    \begin{align*}
      &H_{\sigma_n}(\beta^{F_n})=\log \#E_n=\log N(A_{n,m(n)},n,\varepsilon).
    \end{align*}
    By Lemma 3.1 (3) of \cite{HYZ}, the multi-subadditivity of $H_{\sigma_n}(\beta^\bullet)$, for any finite subset $F\subset G$,
    \begin{align*}
      H_{\sigma_n}(\beta^{F_n})
      &\le \frac{1}{|F|}\sum_{g\in F_n}H_{\sigma_n\circ g^{-1}}(\beta^{F})+|F_n\setminus \{g\in G: F^{-1}g\subseteq F_n\}|\log\#\beta.
    \end{align*}

    Hence
    \begin{align*}
      &\;\;\;\;\frac{1}{|F_n|}H_{\sigma_n}(\beta^{F_n})\\
      &\le \frac{1}{|F|}\frac{1}{|F_n|}\sum_{g\in F_n}H_{\sigma_n\circ g^{-1}}(\beta^{F})+\frac{1}{|F_n|}|F_n\setminus \{g\in G: F^{-1}g\subseteq F_n\}|\log\#\beta\\
      &\le \frac{1}{|F|}H_{\mu_n}(\beta^F)+\frac{|F_n\setminus \{g\in G: F^{-1}g\subseteq F_n\}|}{|F_n|}\log\#\beta.
    \end{align*}

    Let $A$ and $K$ be two finite subsets of $G$ and let $\delta>0$. Recall that the set $A$ is said to be {\it $(K, \delta)$-invariant} if
    $$\frac{|B(A,K)|}{|A|}<\delta,$$ where
    $$B(A,K)=\{g\in G: Kg\cap A\neq\emptyset \text { and } Kg\cap(G\setminus A)\neq\emptyset\}$$ is the {\it $K$-boundary} of $A$.
    An equivalent condition for the sequence of finite subsets $\{F_n\}$ of $G$ to be a F{\o}lner sequence is that for any finite subset $K$ of $G$ and any $\delta>0$, the set $F_n$ is $(K, \delta)$-invariant for all sufficiently large $n$ (see \cite{OW}).

    Denote by $\tilde{F}=F\cup\{e_G\}$. Then we have
    $$F_n\setminus \{g\in G: F^{-1}g\subseteq F_n\}=F_n\cap FF_n^c\subseteq \tilde{F}F_n\cap \tilde{F}F_n^c=B(F_n, \tilde{F}^{-1}).$$
    Thus for any $\delta>0$, if we let $n$ be large enough such that $F_n$ is $(\tilde{F}^{-1}, \delta)$-invariant, then
    $$\frac{|F_n\setminus \{g\in G: F^{-1}g\subseteq F_n\}|}{|F_n|}\le \frac{|B(F_n, \tilde{F}^{-1})|}{|F_n|}<\delta.$$

    Since $\mu(\partial \beta)=0$, we have $\mu(\partial \beta^F)=0$. Letting $n$ tend to infinity,
    \begin{align*}
      \limsup_{n\rightarrow\infty}\frac{1}{|F_n|}\log N(A_{n,m(n)},n,\varepsilon)\le \frac{1}{|F|}H_{\mu}(\beta^F).
    \end{align*}
    This leads to
    \begin{align*}
      \limsup_{n\rightarrow\infty}\frac{1}{|F_n|}\log N(A_{n,m(n)},n,\varepsilon)\le h_{\mu}(X,G),
    \end{align*}
    a contradiction.
\end{proof}
  By the claim, for each $\delta> 0$, there exists $\varepsilon_0>0$ satisfying that for any $0<\varepsilon<\varepsilon_0$, there
exists $M\in \N$ (depending on $\varepsilon$) such that whenever $m>M$, it holds that
\begin{align*}
  \limsup_{n\rightarrow\infty}\frac{1}{|F_n|}\log N(A_{n,m},n,\varepsilon)\le h_{\mu}(X,G)+\frac{\delta}{2}.
\end{align*}
 Let $E_{n,m}$ be a $(F_n,\varepsilon)$-separated set of $A_{n,m}$ with maximal cardinality, then $A_{n,m}\subset \bigcup_{x\in E_{n,m}}B_{F_n}(x,2\varepsilon)$.
 Hence for $s=h_{\mu}(X,G)+2\delta$,
 \begin{align*}
   \mathcal{M}(G_{\mu},s,N,2\varepsilon)&\le \mathcal{M}(\bigcup_{n\ge N}A_{n,m},s,N,2\varepsilon)\\
   &\le \sum_{n\ge N} \sum_{x\in E_{n,m}}\exp(-s|F_n|)\\
   &\le\sum_{n\ge N}\exp((h_{\mu}(X,G)+\delta-s)|F_n|)\\
   &=\sum_{n\ge N}\exp(-\delta|F_n|).
 \end{align*}
 Since $\{F_n\}$ satisfies the condition $\frac{|F_n|}{\log n}\rightarrow \infty$,
 \begin{align*}
   \mathcal{M}(G_{\mu},s,2\varepsilon)\le \lim_{N\rightarrow\infty}\sum_{n\ge N}\exp(-\delta|F_n|)=0,
 \end{align*}
 which implies that $h^B_{top}(G_{\mu},\{F_n\})\le h_{\mu}(X,G)$.

{\bf Acknowledgements}
The research was supported by the National Basic Research Program of China (Grant No. 2013CB834100) and the National Natural
Science Foundation of China (Grant No. 11271191).


\begin{thebibliography}{99}
\bibitem{B} R. Bowen, Topological entropy for noncompact sets, Trans. Amer. Math. Soc. 184 (1973) 125--136.
\bibitem{BK} M. Brin, A. Katok, On Local Entropy, Lecture Notes in Mathematics, vol.1007, Springer, Berlin, 1983, 30--38.
%\bibitem{F} H. Federer, Geometric Measure Theory, Springer-Verlag, New York, 1969.
%\bibitem{FH} D.J. Feng, W. Huang, Variational principles for topological entropies of subsets, J. Funct. Anal. 263(8) (2012) 2228--2254.
\bibitem{HYZ} W. Huang, X. Ye and G. Zhang, Local entropy theory for a countable discrete amenable group action, J. Funct. Anal. 261 (2011), no. 4, 1028--1082.
\bibitem{K} A. Katok, Lyapunov exponents, entropy and periodic orbits for diffeomorphisms, Publ.Math. I.H.E.S., v.51(1980), 137--173.
\bibitem{L} E. Lindenstrauss, Pointwise theorems for amenable groups, Invent. Math. 146 (2001) 259--295.
%\bibitem{M} P. Mattila, Geometry of Sets and Measures in Euclidean Spaces, Cambridge University Press, 1995.
%\bibitem{OP} J.M. Ollagnier, D. Pinchon, The variational principle, Studia Math. 72 (2) (1982) 151--159.
\bibitem{OW} D.S. Ornstein, B. Weiss, The Shannon-McMillan-Breiman theorem for a class of amenable groups, Israel Journal of Mathematics, 44 (1983), no. 1, 53--61.
\bibitem{OW2} D.S. Ornstein, B. Weiss, Entropy and isomorphism theorems for actions of amenable groups, J. Anal. Math. 48(1987) 1--141.
%\bibitem{P} Ya. B. Pesin, Dimension Theory in Dynamical Systems, Contemporary Views and Applications, University of Chicago Press, Chicago, IL, 1997.
\bibitem{PS} C.-E. Pfister and W.G. Sullivan, On the topological entropy of saturated sets, Ergod. Th. Dynam. Sys. 27 (2007), no. 3, 929--956.
\bibitem{S} A. Shulman, Maximal ergodic theorems on groups, Dep. Lit. NIINTI, No.2184, 1988.
%\bibitem{ST} A.M. Stepin, A.T. Tagi-Zade, Variational characterization of topological pressure of the amenable groups of transformations,
Dokl. Akad. Nauk SSSR 254 (3) (1980) 545--549 (in Russian).
%\bibitem{WZ} T. Ward and Q. Zhang, The Abramov-Rokhlin entropy addition formula for amenable group action, Monatshefte f¨¹r Mathematik 114 (1992), no. 3, 317--329.
\bibitem{W} B. Weiss, Actions of amenable groups, Topics in Dynamics and Ergodic Theory. (2003) 226--262. London Math. Soc. Lecture Note Ser., 310, Cambridge Univ. Press,
Cambridge, 2003.
\bibitem{ZC} D. Zheng, E. Chen, Bowen entropy for actions of amenable groups, Israel Journal of Mathematics, accepted.
%\bibitem{ZC2} D. Zheng, E. Chen, On large deviations for amenable group actions,
\end{thebibliography}
\end{document}